\documentclass[psamsfonts]{amsart}

\usepackage[a4paper,bindingoffset=0.0in,left=.975in,right=.975in,top=1.50in,bottom=1.50in,footskip=.35in]{geometry}

\usepackage{amssymb,amsfonts}
\usepackage[all,arc]{xy}
\usepackage{enumerate}
\usepackage{mathrsfs}
\usepackage{amsmath}
\usepackage{setspace}
\usepackage{tikz-cd}

\newtheorem{thm}{Theorem}[section]
\newtheorem{cor}[thm]{Corollary}
\newtheorem{prop}[thm]{Proposition}
\newtheorem{lem}[thm]{Lemma}

\theoremstyle{definition}
\newtheorem{defn}[thm]{Definition}

\newtheorem{notn}[thm]{Notation}

\theoremstyle{remark}
\newtheorem{rem}[thm]{Remark}

\newcommand{\Q}{\Bbb{Q}}

\newcommand{\OO}{\mathcal{O}}

\newcommand{\CC}{\mathcal{C}}

\newcommand{\UU}{\mathcal{U}}

\newcommand{\Z}{\Bbb{Z}}

\newcommand{\N}{\Bbb{N}}

\newcommand{\g}{\frak{g}}

\newcommand{\m}{\frak{m}}

\newcommand{\GSp}{\text{GSp}}

\newcommand{\rig}{\text{rig}}

\newcommand{\Id}{\text{Id}}

\newcommand{\Ind}{\text{Ind}}

\newcommand{\alg}{\text{alg}}

\newcommand{\Sp}{\text{Sp}}

\newcommand{\Hom}{\text{Hom}}

\newcommand{\Lie}{\text{Lie}}

\newcommand{\im}{\text{im}}

\newcommand{\GL}{\text{GL}}

\newcommand{\SL}{\text{SL}}

\newcommand{\y}{\hspace{6pt}}

\makeatletter
\let\c@equation\c@thm
\makeatother
\numberwithin{equation}{section}

\title{Rigid vectors in $p$-adic principal series representations}

\author{Aranya Lahiri and Claus Sorensen}

\date{}

\begin{document}

\begin{abstract}
In this paper we view pro-$p$ Iwahori subgroups $I$ as rigid analytic groups $\Bbb{I}$ for large enough $p$. This is done by endowing $I$ with a natural $p$-valuation, and thereby
generalizing results of Lazard for $\GL_n$.  
We work with a general connected reductive split group over some $p$-adic field (with simply connected derived group) and study the $\Bbb{I}$-analytic vectors in principal series representations. Our main result is an irreducibility criterion which generalizes results of Clozel and Ray in the $\GL_n$-case.
\end{abstract}

\maketitle



\onehalfspacing

\section{Introduction}

In this article we study rigid analytic vectors in continuous $p$-adic prinicpal series representations of a $p$-adic reductive group $G$. Here $G={\bf{G}}(L)$ is the locally profinite group of $L$-rational points of a connected reductive group ${\bf{G}}$ defined over a finite extension $L/\Q_p$. We assume ${\bf{G}}$ is split over $L$, and that the derived group 
${\bf{G}}^{\text{der}}$ is simply connected. Once and for all we choose a Borel subgroup ${\bf{P}}$ over $L$, and a maximal $L$-split torus ${\bf{T}}\subset {\bf{P}}$. Let 
$P={\bf{P}}(L)$ and $T={\bf{T}}(L)$ -- similarly for other linear algebraic groups ${\bf{H}}$ over $L$ we will write $H$ for the group of $L$-points. 

The continuous principal series representations of $G$ are defined as follows, see \cite[Ch.~2]{Sch06}. Start with a continuous character $\chi: T \rightarrow K^{\times}$ taking values in some finite extension $K \supset L$. Inflate $\chi$ to a character of $P$ by letting $\chi=1$ on the unipotent radical as usual, and look at the induction in the space 
$C(G,K)$ of continuous $K$-valued functions on $G$:
$$
\Ind_{P}^G(\chi)=\{f \in C(G,K): f(bg)=\chi(b)f(g)\y \forall b \in P\}.
$$
(Our normalization is different from \cite{Sch06} where $G$ acts by left translations. In our setup $G$ acts by right translations, which means $(g'f)(g)=f(gg')$ for $g,g'\in G$.) One turns $\Ind_{P}^G(\chi)$ into an admissible 
$K$-Banach representation by endowing it with the norm $\|f\|=\sup_{g \in C}|f(g)|$ for some choice of maximal compact subgroup $C \subset G$ for which the Iwasawa decomposition 
$G=PC$ holds.

In his paper \cite{Sch06} Schneider expresses optimism about $\Ind_{P}^G(\chi)$ always having a finite composition series, and he states a precise (topological) irreducibility conjecture 
\cite[Conj.~2.5]{Sch06} -- at least for semisimple and simply connected groups ${\bf{G}}$, in which case the condition on the character is $\chi^{-1}\delta \circ \alpha^{\vee}\notin (\cdot)^{\Z_{>0}}$ for all ${\bf{P}}$-positive roots $\alpha$; that is those $\alpha$ appearing in $\Lie({\bf{P}})$ (and $\delta=\frac{1}{2}\sum_{\alpha>0}\alpha$ as usual). 

The locally analytic analogue of Schneider's irreducibility conjecture was proved by Orlik and Strauch in \cite{OS10}, as a culmination of work by Feaux de Lacroix, Frommer, Kohlhaase, Schneider, Teitelbaum, and others. The actual conjecture is known to be true for $\GL_2(\Q_p)$ but it remains open in general. Ban and Hundley proved the conjecture for
$\chi$ in a certain cone, see \cite[Thm.~1.1]{BH16}.

In this article we restrict $\Ind_{P}^G(\chi)$ to a pro-$p$ Iwahori subgroup $I$ (for $p>\!\!>0$ large enough) and investigate the subspace of rigid analytic vectors 
$\Ind_{P}^G(\chi)_{\Bbb{I}-\text{an}}$ as defined in \cite[Df.~3.3.1]{Eme17}. Here $\Bbb{I}$ is a rigid analytic $\Q_p$-group with $\Bbb{I}(\Q_p)=I$, and as a rigid analytic space $\Bbb{I}$ is a closed ball. 

As a trivial first step the Bruhat-Tits decomposition gives an isomorphism of $I$-representations
$$
\Ind_{P}^G(\chi)_{\Bbb{I}-\text{an}} \overset{\sim}{\longrightarrow} \bigoplus_{w\in W} \Ind_{I \cap wPw^{-1}}^I(w\chi)_{\Bbb{I}-\text{an}} 
$$
where $W=W({\bf{G}},{\bf{T}})$ denotes the Weyl group and $(w\chi)(t)=\chi(w^{-1}tw)$ for $w \in W$. The main problem is to understand each summand $\Ind_{I \cap wPw^{-1}}^I(w\chi)_{\Bbb{I}-\text{an}}$, which by \cite[Prop.~3.3.7]{Eme17} agrees with the rigid analytic induction of $w\chi$ (the space of rigid functions $f:I \rightarrow K$ satisfying the usual transformation property; see (\ref{ind}) in section \ref{slp} plus section \ref{ps} in the main text for more details). Of course, for all these spaces to be non-trivial we must assume the character $\chi$ is rigid analytic when restricted to $T^1=I\cap T$. This means the function $\chi: T^1 \rightarrow K^{\times}\hookrightarrow K$ lies in the space $\CC^{\rig}(T^1,K)$
introduced in definition \ref{rigf}.

The first half of our paper is concerned with promoting $I$ to a rigid analytic group $\Bbb{I}$. Our convention will be that our choice of $I$ corresponds to the 
${\bf{P}}$-negative roots. (This part is potentially confusing; to aid the reader we point out that $P$ is being denoted by $B^-$ in the main body of the text.) See (\ref{iwahori}).
 We show that $I$ admits a natural $p$-valuation $\omega$ provided $p-1>e(L/\Q_p)h_G$ where $h_G$ is the Coxeter number of ${\bf{G}}$ and $e(L/\Q_p)$ is the ramification index;
 see Proposition \ref{char}. Moreover 
 $(I,\omega)$ is saturated, which is what allows us to associate a rigid analytic group $\Bbb{I}$ with $\Q_p$-points $I$. This is presumably well-known but not easily accessible in the literature. We give the details in Section \ref{analytic}. We should emphasize there is some overlap between our work and \cite{CR18}, \cite{Ray20a}. In \cite{CR18} Cornut and Ray give a minimal set of generators for $I$ in any split reductive group ${\bf{G}}$, but they do not directly use $p$-valuations. In \cite{Ray20a} Ray gives an ordered basis for $(I,\omega)$ in the case of $\SL_n(\Z_p)$ and $\GL_n(\Z_p)$. He is working with $\omega$ as defined in \cite[Ex.~23.3]{Sch11} which leads to massive matrix calculations. Part of our original motivation for writing this paper was to give a more conceptual approach to $(I,\omega)$ and its ordered bases for any Chevalley group. 

Having defined the rigid analytic pro-$p$ Iwahori group $\Bbb{I}$ for $p-1>e(L/\Q_p)h_G$ it now makes sense to ask for the $\Bbb{I}$-analytic vectors of any $K$-Banach representation of $I$. Here is our main result for principal series representations, which gives an irreducibility criterion analogous to \cite[Thm.~3.4.12]{OS10} in the rigid setting. 

\begin{thm}\label{main}
Let $G$ be a $p$-adic reductive group over $L$ as above. Assume $p-1>e(L/\Q_p)h_G$. Let $\chi: T \rightarrow K^{\times}$ be a continuous character which is rigid analytic when restricted to $T^1=I \cap T$, and consider the Verma module $\mathcal{V}_{d\chi}$ associated with its derivative $d\chi$ (see Section \ref{verma} for more details). Then 
$\mathcal{V}_{d\chi}$ is simple if and only if the $I$-representation
$$
\Ind_{I \cap wPw^{-1}}^I(w\chi)_{\Bbb{I}-\text{an}} 
$$
is topologically irreducible for some $w \in W$ -- or equivalently all $w \in W$. If so, these representations are mutually non-isomorphic as $w$ varies. 
\end{thm}

There is a standard criterion for the simplicity of $\mathcal{V}_{d\chi}$ due to Bernstein-Gelfand-Gelfand, see \cite[Thm.~7.6.24]{Dix96}. It is the infinitesimal version of the criterion in the irreducibility conjecture discussed above. In Section \ref{symp} we work out the irreducibility criterion explicitly in the case of $\Sp_4(\Q_p)$ for $p>5$. 

The $\GL_2(\Q_p)$-case of Theorem \ref{main} is \cite[Thm.~3.7]{Clo18}, which continued the line of research begun by Robert in \cite{Rob84} and \cite{Rob85}. As part of his thesis Ray extended Clozel's argument to $\GL_n(\Q_p)$ and $w=1$. This is the main result of his paper \cite[Thm.~3.9]{Ray20b}. In the last paragraph of that paper Ray expresses hope that the proof can be adapted to the case $w\neq 1$. 
This is essentially what we do. By reworking the argument of Clozel and Ray, and setting it up in a more streamlined fashion, we are able to avoid the heavy matrix calculations in \cite{Ray20b} and deal with the case $w\neq 1$ with very little extra effort. Here are some of the main differences between \cite{Ray20b} and this article:
\begin{itemize}
\item The coordinates in \cite{Ray20b} are given by the matrix entries (ordered lexicographically). This seems a bit artificial to us, and it results in cumbersome matrix calculations. Instead, for an arbitrary split group, our coordinates arise from a natural ordered basis in the sense of Lazard. This lets us argue more conceptually. 
\item \cite[Lem.~3.4]{Ray20b} gives an explicit Iwahori factorization. We do not need to be explicit. In Remark \ref{promote} we explain how the Iwahori factorization (as $p$-adic manifolds say) can be boosted to an isomorphism of rigid analytic spaces -- without explicitly inverting the multiplication map.
\item The proof in \cite[p.~479]{Ray20b} juggles with $n$ operators  $T_1,\ldots, T_n$ corresponding to the $n$ cocharacters $x \mapsto \text{diag}(1,\ldots,x,\ldots,1)$. We work with a single operator $H_{\mu}$ corresponding to cocharacter $\mu$ adapted to the choice of Weyl element $w$ in the sense that $\langle w\alpha,\mu\rangle>0$ for all positive roots $\alpha$. In the $\GL_n$-case, for $w=1$ our cocharacter $\mu$ can be taken to be $x \mapsto \text{diag}(1,x,x^2, \ldots,x^{n-1})$ for instance. 
\end{itemize}

Our interest in Theorem \ref{main} stems primarily from the globally analytic variant of base change introduced by Clozel in \cite{Clo17} and \cite{Clo18}. In smooth representation theory, base change is reflected by a homomorphism between Hecke algebras. Clozel instead defines a homomorphism between globally analytic distribution algebras such as 
$\mathcal{D}^{\text{an}}(\Bbb{I},K)$ in the notation of \cite[Cor.~5.1.8]{Eme17}, and he seems to suggest that this map should be compatible with classical Langlands base change -- at least for unramified extensions.

As already noted, our Theorem \ref{main} bears a close resemblance to \cite[Thm.~3.4.12]{OS10} in the locally analytic setup. The latter asserts that the strong dual of the locally analytic induction $\big(\Ind_{I \cap wPw^{-1}}^I(w\chi)_{\text{la}}\big)_b'$ is a simple $\mathcal{D}^{\text{la}}(I,K)$-module if $\mathcal{V}_{d\chi}$ is simple. This more or less immediately gives the irreducibility of the induction by employing the duality theory of Schneider and Teitelbaum established in \cite{ST02}, \cite{ST03}. There does not seem to be a good duality theory in the rigid analytic setup. For instance, as pointed out in \cite[App.~A.2]{Clo18} the distribution algebras $\mathcal{D}^{\text{an}}(\Bbb{I},K)$ are non-Noetherian in general, which renders the Fr\'{e}chet-Stein techniques of \cite{ST03} useless in the rigid analytic setup. We tried unsuccessfully to mimic the approach of \cite{OS10}, but Theorem \ref{main} really does seem to require a direct proof along the lines of Clozel and Ray rather than passing to duals.  

 \section{Notation}
 
 Throughout the paper the following notation will remain in force. Let $p>2$ be a prime. We work over a fixed finite extension $L/\Q_p$ with valuation ring $\OO$ and maximal ideal 
 $\m=(\varpi)\subset \OO$. The absolute ramification index $e(L/\Q_p)$ will just be denoted by $e$. We let $\ell=[L:\Q_p]$. 
 
 We fix a connected reductive group ${\bf{G}}$ which is split over $L$, and choose an $L$-split maximal torus ${\bf{T}}\subset {\bf{G}}$. Let 
 $(X^*({\bf{T}}), \Phi, X_*({\bf{T}}), \Phi^{\vee})$ be the associated root datum. Once and for all we pick a system of positive roots $\Phi^+$ with simple roots 
 $\Delta \subset \Phi^+$ and let $\Phi^-=-\Phi^+$. We let $\text{ht}(\cdot)$ be the height function on $\Phi$, and we recall the 
 Coxeter number of ${\bf{G}}$ is $h=1+\text{ht}(\theta)$ where $\theta$ is the highest root. To each $\alpha\in \Phi$ there is attached a unipotent root subgroup 
 ${\bf{U}}_{\alpha}$ normalized by ${\bf{T}}$. 
 
 We are mostly interested in representations of the $p$-adic group $G={\bf{G}}(L)$ with the usual locally profinite topology. Similarly $T={\bf{T}}(L)$ and we have the maximal compact subgroup $T^0$ and its Sylow pro-$p$-subgroup $T^1$. We let $U_{\alpha}={\bf{U}}_{\alpha}(L)$ and note that there is an isomorphism 
 $u_{\alpha}: L \overset{\sim}{\longrightarrow} U_{\alpha}$ such that $tu_{\alpha}(x)t^{-1}=u_{\alpha}(\alpha(t)x)$ for all $t \in T$ and $x \in L$. This gives a filtration of $U_{\alpha}$ by the subgroups $U_{\alpha,r}=u_{\alpha}(\m^r)$ for varying $r \in \Z$.

We consider the full Iwahori subgroup $J$ corresponding to $\Phi^+$. Multiplication defines a homeomorphism 
$$
\prod_{\alpha\in \Phi^-}U_{\alpha,1} \times T^0 \times \prod_{\alpha\in \Phi^+}U_{\alpha,0} \overset{\sim}{\longrightarrow} J.
$$
Our main interest however is the pro-$p$ Iwahori subgroup $I$ which is the Sylow pro-$p$ subgroup of $J$. Again multiplication gives a homeomorphism
\begin{equation}\label{iwahori}
\prod_{\alpha\in \Phi^-}U_{\alpha,1} \times T^1 \times \prod_{\alpha\in \Phi^+}U_{\alpha,0} \overset{\sim}{\longrightarrow} I.
\end{equation}
Here the products over $\Phi^-$ and $\Phi^+$ are ordered in an arbitrarily chosen way. For the latter see the proof of \cite[Lem.~2.3]{OS19} and Lemma 2.1, part i, in \cite{OS19}.

\section{A $p$-valuation on pro-$p$ Iwahori}

In this section we will define a $p$-valuation $\omega$ on $I$ (in the sense of Lazard) provided $p-1>eh$ and exhibit an ordered basis for $(I,\omega)$. The strategy is to emulate \cite[Ex.~23.3]{Sch11}
and conjugate $I$ into a small enough principal congruence subgroup of ${\bf{G}}(E)$ for a large enough finite extension $E/L$. This approach is based on the following well-known observation.

\begin{lem}
Let $E/\Q_p$ be a finite extension and suppose $\mathcal{G}$ is an $\OO_E$-subgroup scheme of $\GL_N$ for some $N$. For $r=1,2,3,\ldots$ consider the congruence subgroup
 $K_r$ defined as 
$$
K_r=\ker\big(\mathcal{G}(\OO_E)\longrightarrow \mathcal{G}(\OO_E/\m_E^r)\big).
$$ 
Then for $r>e_E/(p-1)$ the expression
$$
\omega(g)=\frac{1}{e_E}\cdot \sup\{n \in \N: g \in K_n\}
$$ 
defines a $p$-valuation on $K_r$. (Here $e_E=e(E/\Q_p)$ and $\OO_E$ denotes the valuation ring of $E$; by convention $\omega(1)=\infty$.) 
\end{lem}

\begin{proof}
The $\GL_N$-case is \cite[Ex.~23.2]{Sch11}. There $\omega$ is defined by $\omega(g)=\min_{i,j} v(a_{ij}-\delta_{ij})$ for a matrix $g=(a_{ij})$ where $v$ is the additive valuation on
$E$ with $v(p)=1$, but this is easily checked to agree with our definition. For a subgroup scheme $\mathcal{G}\hookrightarrow \GL_N$ just note that $K_r=\mathcal{G}(\OO_E)\cap \tilde{K}_r$ where $\tilde{K}_r$ denotes the $r$th congruence subgroup of $\GL_N$. 
\end{proof}

We will apply this to the following $\mathcal{G}$. As in \cite[Sect.~2.1.6]{OS19} we fix a hyperspecial vertex $x_0$ in the apartment of the semisimple Bruhat-Tits building of $G$ corresponding to ${\bf{T}}$. We consider the associated smooth affine group scheme $\mathcal{G}_{x_0}$ over $\OO$ and its identity component $\mathcal{G}_{x_0}^{\circ}$.
As discussed in \cite[Sect.~7.2.2]{OS19} the $r^{\text{th}}$ congruence subgroup of $\mathcal{G}_{x_0}^{\circ}$ is generated by all the $U_{\alpha,r}$ for varying $\alpha \in \Phi$, and a subgroup $T^r\subset T$ defined in terms of the Neron model of ${\bf{T}}$. See \cite[Prop.~7.9]{OS19}. 

Instead of $\mathcal{G}_{x_0}^{\circ}$ we work with the base change $\mathcal{G}=\mathcal{G}_{x_0}^{\circ} \times_{\text{Spec}(\OO)}\text{Spec}(\OO_E)$, where $E/L$ is a yet unspecified finite extension. As above we let $K_r$ denote the $r^{\text{th}}$ congruence subgroup of $\mathcal{G}(\OO_E)$. 

\begin{lem}\label{Et}
Assume $p-1>eh$. Then there exists a pair $(E,t)$ consisting of a finite extension $E/L$ and an element $t \in {\bf{T}}(E)$ such that $tIt^{-1}\subset K_r$ where 
$r$ is the smallest integer $>e_E/(p-1)$. 
\end{lem}

\begin{proof}
Let $E/L$ be arbitrary for now. We apply \cite[Prop.~7.9]{OS19} to the base change ${\bf{G}}\times_{\text{Spec}(L)} \text{Spec}(E)$ and the facet $\{x_0\}$. 
This tells us the multiplication map gives a homeomorphism
$$
\prod_{\alpha\in \Phi^-}U_{\alpha,r}^E \times T_E^r \times \prod_{\alpha\in \Phi^+}U_{\alpha,r}^E \overset{\sim}{\longrightarrow} K_r.
$$
Here the superscripts/subscript $E$ indicate we are are now working over $E$. For instance $U_{\alpha,r}^E=u_{\alpha}^E(\m_E^r)$ where
$u_{\alpha}^E: E \overset{\sim}{\longrightarrow} {\bf{U}}_{\alpha}(E)$ extends $u_{\alpha}$. By conjugating the analogous factorization (\ref{iwahori}) of $I$ by $t \in {\bf{T}}(E)$ we see that the inclusion 
$tIt^{-1}\subset K_r$ holds if and only if

\begin{itemize}
\item[1.] For $\alpha \in \Phi^-$,
$$
tU_{\alpha,1}t^{-1}\subset U_{\alpha,r}^E \Longleftrightarrow \alpha(t)\m \subset \m_E^r \Longleftrightarrow -v(\alpha(t))<\frac{1}{e}-\frac{1}{p-1}.
$$
\item[2.] For the middle factor,
$$
tT^1t^{-1}=T^1 \subset T_E^r \Longleftrightarrow \m\subset \m_E^r \Longleftrightarrow p-1>e.
$$
\item[3.] For $\alpha \in \Phi^+$,
$$
tU_{\alpha,0}t^{-1}\subset U_{\alpha,r}^E \Longleftrightarrow \alpha(t)\OO \subset \m_E^r \Longleftrightarrow v(\alpha(t))> \frac{1}{p-1}.
$$
\end{itemize}

We remind the reader that $v(p)=1$. These inequalities can be summarized as saying that
$$
\frac{1}{p-1}<v(\alpha(t))<\frac{1}{e}-\frac{1}{p-1}
$$
for all $\alpha \in \Phi^+$. It remains to construct a pair $(E,t)$ satisfying all these inequalities. We will do this by means of an auxiliary cocharacter defined as follows. Start by enlarging $\Delta$ to a $\Q$-basis for $X^*({\bf{T}})_{\Q}$ and choose a linear form $X^*({\bf{T}})_{\Q}\rightarrow \Q$ sending all elements of $\Delta$ to one. This is given by pairing with some element $\mu_0 \in X_*({\bf{T}})_{\Q}$. Thus $\langle \alpha,\mu_0 \rangle=\text{ht}(\alpha)$ for all $\alpha \in \Phi$. Pick an $a \in \N$ for which $a\mu_0 \in X_*({\bf{T}})$ and call this cocharacter $\mu$. Now suppose $E/L$ is large enough to contain an element $c \neq 0$ with $v(c)=\frac{1}{aeh}$. We claim $t=\mu(c)$ works. Indeed 
$$
v(\alpha(t))=v(c^{\langle\alpha,\mu\rangle})=\langle\alpha,\mu\rangle v(c)=a\text{ht}(\alpha)\cdot \frac{1}{aeh}=\frac{\text{ht}(\alpha)}{eh}
$$ 
and $1\leq \text{ht}(\alpha)\leq h-1$ for all $\alpha \in \Phi^+$. It therefore suffices to check that 
$$
\frac{1}{p-1}<\frac{1}{eh}<\frac{1}{h-1}\cdot \bigg(\frac{1}{e}-\frac{1}{p-1}\bigg),
$$
and as one readily verifies this amounts to our assumption that $p-1>eh$. 
\end{proof}

Combining the two preceding lemmas endows $I$ with a $p$-valuation $\omega$ by restricting and conjugating the one on $K_r$. We wish to give an intrinsic description of 
$\omega$ independent of the choice of $(E,t)$. First we introduce the following notion.

\begin{defn}\label{compatible}
Let $H$ be a $p$-valuable group and let $(H_1,\ldots,H_s)$ be a list of subgroups such that the multiplication map 
$$
m: H_1 \times \cdots \times H_s \longrightarrow H
$$
is bijective. We say that a $p$-valuation $\omega$ on $H$ is compatible with the tuple $(H_1,\ldots,H_s)$ if 
$$
\omega(h_1\cdots h_s)=\min\{\omega(h_1),\ldots, \omega(h_s)\}
$$
for all $h_i \in H_i$ with $i=1,\ldots,s$. 
\end{defn}

In this situation one can construct an ordered basis for $(H,\omega)$ by concatenating ordered bases for each factor $(H_i,\omega|_{H_i})$ as $i=1,\ldots,s$. For the definition of an ordered basis of a $p$-valued group we refer the reader to \cite[p.~182]{Sch11}.

Here is our characterization of the $p$-valuation we have defined on $I$.

\begin{prop}\label{char}
Assume $p-1>eh$. Then there exists a unique $p$-valuation $\omega$ on $I$ satisfying the following properties:
\begin{itemize}
\item[(a)] $\omega$ is compatible with the Iwahori factorization (\ref{iwahori}) of $I$, cf. Definition \ref{compatible}.
\item[(b)] $\omega(u_{\alpha}(x))=v(x)+\frac{\text{ht}(\alpha)}{eh}$ where
$\begin{cases}
\text{$x\in \m$ if $\alpha \in \Phi^-$} \\
\text{$x \in \OO$ if $\alpha \in \Phi^+$.} 
\end{cases}$
\item[(c)] $\omega(t)=\frac{1}{e}\cdot \sup\{n \in \N: t \in T^n\}$ for $t \in T^1$.
\end{itemize}
(Recall that $T^n$ denotes the $n^{\text{th}}$ congruence subgroup of $T$; in part (a) the compatibility of $\omega$ holds for any ordering of the roots
$\Phi^{\pm}$.) Moreover $(I,\omega)$ is saturated (i.e., every $g \in I$ with $\omega(g)\geq \frac{p}{p-1}$ lies in $I^p$).  
\end{prop}

\begin{proof}
Uniqueness is clear (as $\omega$ is determined by its restriction to each factor of the Iwahori decomposition). We need to show the $\omega$ we defined satisfies (a) through (c). 
Fix $(E,t)$ as in the proof of Lemma \ref{Et}. Then $\omega(g)=\frac{1}{e_E}\cdot \sup\{n \in \N: tgt^{-1}\in K_n\}$ for $g \in I$. Consider the Iwahori factorization
$$
g=\bigg(\prod_{\alpha \in \Phi^-}u_{\alpha}\bigg) \cdot t_1\cdot \bigg(\prod_{\alpha \in \Phi^+}u_{\alpha}\bigg)
$$
according to (\ref{iwahori}). Conjugating by $t$ we find that $tgt^{-1}\in K_n$ if and only if all $tu_{\alpha}t^{-1}\in U_{\alpha,n}^E$ and $t_1 \in T_E^n$. The largest $n$ with this property is the minimum of $\min_{\alpha \in \Phi}\omega(u_{\alpha})$ and $\omega(t_1)$. This shows (a). The formula in (b) follows almost immediately from the relation $v(\alpha(t))=\frac{\text{ht}(\alpha)}{eh}$ as noted in the proof of Lemma \ref{Et}. Since ${\bf{T}}$ is split the proof of (c) easily reduces to the $\GL_1$-case. When $t_1 \in 1+\m$ just note that the largest $n$ for which $t_1 \in 1+\m_E^n$ is $n=e(E/L)v_L(t_1-1)$. In other words $\omega(t_1)=\frac{1}{e}v_L(t_1-1)$ as claimed. (Here $v_L$ denotes the additive valuation on $L$ with $v_L(\varpi)=1$.)

To show saturation it suffices to find an ordered basis consisting of elements with $p$-valuation $\leq \frac{p}{p-1}$, see
\cite[Prop.~26.11]{Sch11}. We start by writing down a basis for $U_{\alpha,0}$ for all $\alpha \in \Phi^+$. Pick a $\Z_p$-basis 
$(b_1,\ldots,b_{\ell})$ for $\OO$. Then $(u_{\alpha}(b_1),\ldots, u_{\alpha}(b_{\ell}))$ is a basis for $U_{\alpha,0}$. Indeed any element $u \in U_{\alpha,0}$ can be factored uniquely as
$$
u=u_{\alpha}(x)=\prod_{i=1}^{\ell} u_{\alpha}(b_i)^{x_i}
$$
where $x=\sum_{i=1}^{\ell} x_ib_i$ with coefficients $x_i \in \Z_p$. To verify that $\omega(u)=\min_{i=1,\ldots,\ell}\{v(x_i)+\omega(u_{\alpha}(b_i))\}$ it suffices to note the relations
$v(x)=\min_{i=1,\ldots,\ell}\{v(x_i)\}$ and $v(b_i)=0$. Minor modifications of the argument when $\alpha \in \Phi^-$ show that $U_{\alpha,1}$ has basis 
$(u_{\alpha}(\varpi b_1),\ldots, u_{\alpha}(\varpi b_{\ell}))$. Finally, to find a basis for $T^1$ we first consider the $\GL_1$-case $1+\m$. Since $p-1>e$ the exp and log give
isomorphisms $\m^n \overset{\sim}{\longrightarrow} 1+\m^n$ for all $n \geq 1$. It follows easily that $(e^{\varpi b_1},\ldots,e^{\varpi b_{\ell}})$ is a basis for $1+\m$. (Here $e^{\varpi b_i}=\exp(\varpi b_i)$ and $e$ should not be confused with the ramification index.)
In general pick a basis $\mathcal{B}$ for $X_*({\bf{T}})$ and concatenate all tuples $(\mu(e^{\varpi b_1}),\ldots,\mu(e^{\varpi b_{\ell}}))$ as $\mu \in \mathcal{B}$ varies. On these basis elements the $p$-valuation takes the values
\begin{itemize}
\item $\omega(u_{\alpha}(\varpi b_i))=\frac{1}{e}+\frac{\text{ht}(\alpha)}{eh}$ for $\alpha \in \Phi^-$.
\item $\omega(\mu(e^{\varpi b_i}))=\frac{1}{e}$ for  $\mu \in \mathcal{B}$.
\item $\omega(u_{\alpha}(b_i))=\frac{\text{ht}(\alpha)}{eh}$ for $\alpha \in \Phi^+$.
\end{itemize}
We just need to check all these numbers are at most $\frac{p}{p-1}$. For example $\frac{\text{ht}(\alpha)}{eh}\leq \frac{h-1}{eh}<\frac{1}{e}\leq 1<\frac{p}{p-1}$ for all 
positive roots $\alpha$, and $\frac{1}{e}+\frac{\text{ht}(\alpha)}{eh}\leq \frac{1}{e}-\frac{1}{eh}=\frac{h-1}{eh}<\frac{p}{p-1}$ also for negative roots $\alpha$.
\end{proof}

The observation that $(I,\omega)$ is saturated will be crucial in the next section when we attach a rigid analytic group to this pair. 

\begin{rem}
For $G=\GL_n(\Q_p)$ the condition $p-1>eh$ is also necessary for the existence of a $p$-valuation. In this example $e=1$ and $h=n$, and when $p\leq n+1$ there is $p$-torsion in $I$. 
To see this, for $p \leq n$ just consider a permutation matrix given by a $p$-cycle. A conjugate thereof lies in $I$. For $p=n+1$ consider multiplication by 
$\zeta_p$ on the free rank $n$ module $\Z_p[\zeta_p]$ over $\Z_p$. Again this gives an element of order $p$ in $\GL_n(\Z_p)$ which can be conjugated into $I$ because the latter is a Sylow pro-$p$ subgroup. (Since a $p$-valuation satisfies $\omega(g^p)=\omega(g)+1$ a $p$-valuable group is torsion-free.)
\end{rem}

For future reference we want to record the following consequence of the proof of Proposition \ref{char}.

\begin{cor}
Assume $p-1>eh$ and pick a $\Z_p$-basis $(b_1,\ldots,b_{\ell})$ for $\OO$. If ${\bf{G}}$ is semisimple and simply connected
the following is an ordered basis for $(I,\omega)$.
\begin{itemize}
\item $(u_{\alpha}(\varpi b_1),\ldots, u_{\alpha}(\varpi b_{\ell}))_{\alpha \in \Phi^-}$,
\item $(\alpha^{\vee}(e^{\varpi b_1}),\ldots,\alpha^{\vee}(e^{\varpi b_{\ell}}))_{\alpha \in \Delta}$,
\item $(u_{\alpha}(b_1),\ldots, u_{\alpha}(b_{\ell}))_{\alpha \in \Phi^+}$.
\end{itemize}
All basis elements have $p$-valuation at most $\frac{1}{e}$.
\end{cor}

\begin{proof}
This was shown above. When ${\bf{G}}$ is semisimple and simply connected one can take $\mathcal{B}=\Delta^{\vee}$.
\end{proof}

\section{Rigid analytic groups}\label{analytic}

In this section $(H,\omega)$ denotes an arbitrary $p$-valued group equipped with an ordered basis $(h_1,\ldots,h_d)$.
Thus $(x_1,\ldots,x_d)\mapsto h_1^{x_1}\cdots h_d^{x_d}$ gives a homeomorphism $\Z_p^d \rightarrow H$ and $\omega(h)=\min_{i=1,\ldots,d}\{v(x_i)+\omega(h_i)\}$
where $h \in H$ is factored as $h_1^{x_1}\cdots h_d^{x_d}$. 

We will review how to associate an affinoid rigid analytic group over $\Q_p$ to a saturated pair $(H,\omega)$. The underlying rigid analytic space is nothing but the 
$d$-dimensional closed ball $\text{Sp} (\Q_p\langle Z_1,\ldots Z_d \rangle)$. Here $\Q_p\langle Z_1,\ldots Z_d \rangle$ denotes the Tate algebra in $d$ variables over $\Q_p$. 

\begin{prop}\label{riggrp}
Let $(H,\omega)$ be a saturated $p$-valued group with basis $(h_1,\ldots,h_d)$. Then there exists a unique rigid analytic group
$\Bbb{H}$ over $\Q_p$ such that
\begin{itemize}
\item[(a)] $\Bbb{H}=\text{Sp} (\Q_p\langle Z_1,\ldots Z_d \rangle)$ as a rigid analytic space;
\item[(b)] the map $\varphi \mapsto h_1^{\varphi(Z_1)}\cdots h_d^{\varphi(Z_d)}$ gives an isomorphism of abstract groups
$$
\Bbb{H}(\Q_p)=\Hom_{\Q_p-\text{alg}}(\Q_p\langle Z_1,\ldots Z_d \rangle,\Q_p) \overset{\sim}{\longrightarrow} H.
$$
\end{itemize}
\end{prop}

\begin{proof}
We construct $\Bbb{H}$. The multiplication map $\Bbb{H}\times_{\text{Sp}(\Q_p)} \Bbb{H} \longrightarrow \Bbb{H}$ arises from comultiplication on the Tate algebra 
$\Q_p\langle Z_1,\ldots Z_d \rangle$ defined as follows. By \cite[Prop.~29.2]{Sch11} and \cite[Rem.~29.3]{Sch11} there exists power series 
$F_i \in \Z_p\langle X_1,\ldots X_d, Y_1,\ldots Y_d  \rangle$ for $i=1,\ldots,d$ such that
$$
(h_1^{x_1}\cdots h_d^{x_d})\cdot (h_1^{y_1}\cdots h_d^{y_d})=h_1^{F_1(x,y)}\cdots h_d^{F_d(x,y)}
$$
for all tuples $x=(x_1,\ldots,x_d)$ and $y=(y_1,\ldots,y_d)$ in $\Z_p^d$. The fact that $F_i$ has coefficients in $\Z_p$ (not just $\Q_p$) is due to the assumption that $(H,\omega)$ is saturated, and this allows us to define the map
$$
\Q_p\langle Z_1,\ldots Z_d \rangle \longrightarrow  \Q_p\langle X_1,\ldots X_d\rangle \widehat{\otimes} \Q_p\langle Y_1,\ldots Y_d  \rangle=\Q_p\langle X_1,\ldots X_d, Y_1,\ldots Y_d  \rangle
$$
by sending $Z_i \mapsto F_i(X,Y)$ where $X=(X_1,\ldots,X_d)$ and similarly for $Y$. Since $(F_1,\ldots,F_d)$ is a formal group law one quickly verifies this gives a rigid analytic group structure on $\Bbb{H}=\text{Sp} (\Q_p\langle Z_1,\ldots Z_d \rangle)$. (The identity $\Sp(\Q_p)\rightarrow \Bbb{H}$ comes from the co-unit on the Tate algebra and we leave it to the reader to construct the inversion map $\Bbb{H} \rightarrow \Bbb{H}$ along the same lines.)

To show this $\Bbb{H}$ satisfies (b) let $\varphi_1,\varphi_2 \in \Bbb{H}(\Q_p)$. Unwinding definitions it suffices to note that
\begin{equation}\label{point}
(\varphi_1 \cdot \varphi_2)(Z_i)=F_i\big(\varphi_1(Z_1),\ldots,\varphi_1(Z_d), \varphi_2(Z_1),\ldots,\varphi_2(Z_d)\big),
\end{equation}
but this is immediate in light of how multiplication on $\Bbb{H}(\Q_p)$ relates to comultiplication.

Finally, suppose $\text{Sp} (\Q_p\langle Z_1,\ldots Z_d \rangle)$ carries some rigid analytic group structure for which (b) holds. This structure is given by power series
$G_i \in \Z_p\langle X_1,\ldots X_d, Y_1,\ldots Y_d  \rangle$ for $i=1,\ldots,d$ satisfying the analogue of (\ref{point}) -- still with $F_i$ on the right-hand side, but with 
$\varphi_1 \cdot \varphi_2$ defined relative to the $G_i$. This leads to the equality
$$
G_i\big(\varphi_1(Z_1),\ldots,\varphi_1(Z_d), \varphi_2(Z_1),\ldots,\varphi_2(Z_d)\big)=F_i\big(\varphi_1(Z_1),\ldots,\varphi_1(Z_d), \varphi_2(Z_1),\ldots,\varphi_2(Z_d)\big)
$$
Since $\varphi_1$ and $\varphi_2$ are arbitrary we infer that
$F_i$ and $G_i$ define the same function $\Z_p^d \times \Z_p^d\rightarrow \Z_p$, and therefore $F_i=G_i$ for all $i$, cf. \cite[Cor.~5.8]{Sch11} with $\varepsilon=1$ and $V=\Q_p$.\end{proof}

The next result shows how this construction behaves under a change of basis.

\begin{prop}\label{unique}
Let $(H,\omega)$ be a saturated group with bases $(h_1,\ldots,h_d)$ and $(h_1',\ldots,h_d')$.
Let $\Bbb{H}$ and $\Bbb{H}'$ denote the rigid analytic groups associated with these bases as in Proposition \ref{riggrp}. There is a unique isomorphism of rigid analytic groups 
$\Bbb{H} \overset{\sim}{\longrightarrow} \Bbb{H}'$ such that the induced map on $\Q_p$-points makes the following diagram commute:
\[
\begin{tikzcd}
    \Bbb{H}(\Q_p) \arrow{rr}{\sim} \arrow[swap]{dr}{} & &  \Bbb{H}'(\Q_p)  \arrow{dl}{} \\[10pt]
    & H
\end{tikzcd}
\]
\end{prop}

\begin{proof}
Uniqueness is a straightforward consequence of the fact that $\Bbb{H}(\Q_p)$ is Zariski dense in $\Bbb{H}$. (An automorphism of 
$\Bbb{H}$ inducing the identity on $\Bbb{H}(\Q_p)$ must be the identity.) 

To define the desired isomorphism we compare the formal group law $\underline{F}=(F_1,\ldots,F_d)$ from the proof of Proposition \ref{riggrp} to the formal group law
$\underline{F}'=(F_1',\ldots,F_d')$ arising from $(h_1',\ldots,h_d')$. Again by \cite[Prop.~29.2]{Sch11} and \cite[Rem.~29.3]{Sch11} there are $d$ power series 
$\Phi_i \in \Z_p \langle X_1,\ldots X_d \rangle$ such that
$$
h_1^{x_1}\cdots h_d^{x_d}=(h_1')^{\Phi_1(x)}\cdots (h_d')^{\Phi_d(x)}
$$
for all $x=(x_1,\ldots,x_d)\in \Z_p^d$. Similarly with the bases interchanged. This gives an isomorphism of formal group laws $\underline{F} \overset{\sim}{\longrightarrow} \underline{F}'$,
see \cite[p.~143]{Sch11}. Mapping $Z_i'\mapsto \Phi_i(Z_1,\ldots,Z_d)$ then yields an isomorphism 
$$
\Q_p\langle Z_1',\ldots Z_d' \rangle\overset{\sim}{\longrightarrow}  \Q_p\langle Z_1,\ldots Z_d \rangle
$$
which intertwines the two comultiplications. Passing to spectra we finally get an isomorphism of rigid analytic groups $\Bbb{H} \overset{\sim}{\longrightarrow} \Bbb{H}'$. The commutativity of the diagram essentially just boils down to the definition of  the $\Phi_i$.
\end{proof}

Let us revisit the situation in Definition \ref{compatible}.

\begin{rem}\label{promote}
Suppose every $(H_i,\omega|_{H_i})$ is saturated, where $i=1,\ldots,s$. Then so is $(H,\omega)$ and one can obtain a basis for the latter by stringing together bases
$$
(h_{11},\ldots,h_{1d_1}) \y \ldots \y (h_{s1},\ldots,h_{sd_s})
$$
for each subgroup. As in Proposition \ref{riggrp} this gives rigid analytic subgroups $\Bbb{H}_1,\ldots,\Bbb{H}_s$ of $\Bbb{H}$ over $\Q_p$, and multiplication in $\Bbb{H}$ gives an isomorphism of rigid analytic spaces 
$$
\Bbb{H}_1\times_{\Sp(\Q_p)} \cdots \times_{\Sp(\Q_p)} \Bbb{H}_s \overset{\sim}{\longrightarrow} \Bbb{H}
$$ 
which on $\Q_p$-points corresponds to the bijection $H_1 \times \cdots \times H_s \longrightarrow H$. In this sense the factorization of $H$ as a $p$-adic manifold can be promoted to an isomorphism of rigid analytic spaces. 
\end{rem}

We end this section by introducing the space of rigid functions on a saturated $p$-valued group $H$. These functions take values in a fixed finite extension $K/\Q_p$. Choosing a basis 
$(h_1,\ldots,h_d)$ gives a rigid analytic group $\Bbb{H}$, and a $K$-valued rigid analytic function on $\Bbb{H}$ naturally defines a continuous function $f:H \rightarrow K$. In more detail, an element $\sum_I c_I Z_1^{i_1}\cdots Z_d^{i_d}$ of the global sections $\Gamma(\Bbb{H},\OO_{\Bbb{H}})\otimes_{\Q_p}K=K\langle Z_1,\ldots Z_d \rangle$ gives the function 
$f(h)=\sum_I c_I x_1^{i_1}\cdots x_d^{i_d}$ where $h=h_1^{x_1}\cdots h_d^{x_d}$. (Here and elsewhere\footnote{There should be no risk of confusing the tuple $I$ with the pro-$p$ Iwahori.} in this article $I$ runs over tuples $I=(i_1,\ldots,i_d)\in \N_0^d$ and the coefficients $c_I$ form a null-sequence in $K$ as $|I|=i_1+\cdots+i_d\rightarrow \infty$.) By \cite[Cor.~5.8]{Sch11} for instance, the power series is determined by the function $f$.
The image of the map $\Gamma(\Bbb{H},\OO_{\Bbb{H}})\otimes_{\Q_p}K \hookrightarrow \CC(H,K)$ is our space of rigid analytic functions $\CC^{\rig}(H,K)$. 
For this to make sense we need to check that this space is independent of the choice of basis, so suppose $(h_1',\ldots,h_d')$ is another basis with corresponding rigid analytic group $\Bbb{H}'$ as above. The isomorphism $\Bbb{H} \overset{\sim}{\longrightarrow} \Bbb{H}'$ in Proposition \ref{unique} amounts to an isomorphism
\[
\begin{tikzcd}
   \Gamma(\Bbb{H}',\OO_{\Bbb{H}'})\otimes_{\Q_p}K  \arrow{rr}{\sim} \arrow[swap]{dr}{} & &  \Gamma(\Bbb{H},\OO_{\Bbb{H}})\otimes_{\Q_p}K  \arrow{dl}{} \\[10pt]
    & \CC(H,K)
\end{tikzcd}
\]
compatible with comultiplication. The diagram commutes so the images of the two slanted arrows coincide. Furthermore, the isomorphism in the diagram preserves the Gauss norms on the two Tate algebras (in the notation of the proof of Proposition \ref{unique}  one easily checks $\|\Phi_i\|=1$). Thus $\CC^{\rig}(H,K)$ even carries a canonical norm. Let us summarize this.

\begin{defn}\label{rigf}
The space $\CC^{\rig}(H,K)$ consists of all functions $f:H \rightarrow K$ which can be expanded as
$$
f(h_1^{x_1}\cdots h_d^{x_d})=\sum_I c_I x_1^{i_1}\cdots x_d^{i_d}
$$
relative to some or equivalently any basis $(h_1,\ldots,h_d)$ where $c_I \rightarrow 0$ as $|I|\rightarrow \infty$. Under pointwise multiplication $\CC^{\rig}(H,K)$ becomes a $K$-Banach algebra when equipped with the Gauss norm
$$
\|f\|={\sup}_I |c_I|.
$$
Here $|p|=1/p$. The Gauss norm $\|\cdot\|$ is independent of the choice of basis for $(H,\omega)$.
\end{defn}

The group $H$ acts on $\CC(H,K)$ by right (and left) translations. 

\begin{lem}
The $H$-action preserves the subspace $\CC^{\rig}(H,K)$ and the Gauss norm is $H$-invariant.
\end{lem}

\begin{proof}
We right translate $f \in \CC^{\rig}(H,K)$ by an element $h_0=h_1^{a_1}\cdots h_d^{a_d}$. The function $h_0f$ has the form 
$$
(h_0f)(h_1^{x_1}\cdots h_d^{x_d})=f(h_1^{x_1}\cdots h_d^{x_d}\cdot h_1^{a_1}\cdots h_d^{a_d})=\sum_I c_I F_1(x,a)^{i_1}\cdots F_d(x,a)^{i_d}
$$
where $a=(a_1,\ldots,a_d)$. Recall the power series $F_i$ introduced in the proof of Proposition \ref{riggrp} have coefficients in $\Z_p$. In particular 
$\sum_I c_I F_1(X,a)^{i_1}\cdots F_d(X,a)^{i_d}$ converges in $K\langle X_1,\ldots X_d \rangle$ and induces $h_0f$, which is therefore in $\CC^{\rig}(H,K)$.
In addition, this argument shows $\|h_0f\|\leq \|f\|$. Since $h_0$ and $f$ are arbitrary we must have equality.
 \end{proof}

\section{Weyl conjugates}

Let $W$ be the Weyl group of $(G,T)$. 
Recall, as noted in the proof of \cite[Lem.~2.3]{OS19} for instance, that multiplication defines an injection
$$
\prod_{\alpha \in \Phi^-} U_{\alpha}\times T \times \prod_{\alpha \in \Phi^+} U_{\alpha} \hookrightarrow G
$$
where the products are ordered arbitrarily. We conjugate it by $w \in W$ (or rather a representative in the normalizer of $T$) which shows the multiplication map
$$
\prod_{\alpha \in \Phi^-} U_{w\alpha}\times T \times \prod_{\alpha \in \Phi^+} U_{w\alpha} \hookrightarrow G
$$
remains an injection. Here $U_{w\alpha}=wU_{\alpha}w^{-1}$ and $(w\alpha)(t)=\alpha(w^{-1}tw)$. 

\begin{defn}
For $w \in W$ introduce the two subgroups
\begin{itemize}
\item $U_w^-=\prod_{\alpha \in \Phi^-} (I \cap U_{w\alpha})$
\item $U_w^+=\prod_{\alpha \in \Phi^+} (I \cap U_{w\alpha})$
\end{itemize}
for some ordering of $\Phi^-$ and $\Phi^+$. Note that $I \cap U_{w\alpha}=\begin{cases}
\text{$U_{w\alpha,1}$ if $w\alpha \in \Phi^-$} \\
\text{$U_{w\alpha,0}$ if $w\alpha \in \Phi^+$.} 
\end{cases}$
\end{defn}

\begin{lem}\label{weyl}
Assume the derived group ${\bf{G}}^{\text{der}}$ is simply connected. Then the multiplication map gives a homeomorphism
$$
U_w^- \times T^1 \times U_w^+ \overset{\sim}{ \longrightarrow} I
$$
for all $w \in W$. In particular $U_w^{\pm}=I \cap w\big(\prod_{\alpha \in \Phi^{\pm}}U_{\alpha}\big)w^{-1}$. Furthermore, when $p-1>eh$ the $p$-valuation $\omega$ on $I$ from Proposition \ref{char}
is compatible with the above factorization for all $w$ (see Definition \ref{compatible}). 
\end{lem}

\begin{proof}
The previous discussion shows the multiplication map
$$
\prod_{\alpha \in \Phi^-} (I \cap U_{w\alpha}) \times T^1 \times \prod_{\alpha \in \Phi^+} (I \cap U_{w\alpha}) \longrightarrow I
$$
is injective. It is also surjective, as we now explain. Since we are assuming ${\bf{G}}^{\text{der}}$ is simply connected, Bruhat-Tits theory shows that the full Iwahori subgroup $J \supset I$ has a factorization
$$
T^0 \times \prod_{\alpha \in \Phi} (J \cap U_{\alpha})\overset{\sim}{ \longrightarrow} J
$$
for any ordering of $\Phi$, see \cite[Lem.~3.3.2]{OS10} and the references given in the proof thereof. For full disclosure those references to \cite[Sect.~3.1.1]{Tit79} require ${\bf{G}}$ to be semisimple and simply connected, but one can easily reduce the general reductive case to the semisimple case by observing that all root groups lie in the derived group.

Clearly $J \cap U_{\alpha}$ are pro-$p$ groups (subgroups of 
$\m^{-N}$ for large enough $N$) so  $J \cap U_{\alpha} \subset I$ and we therefore also have the analogous factorization of pro-$p$ Iwahori $I$ itself
$$
T^1 \times \prod_{\alpha \in \Phi} (I \cap U_{\alpha})\overset{\sim}{ \longrightarrow} I.
$$
We order $\Phi$ as follows. The first batch of roots is $w\Phi^-$ (in some order) and the second is $w\Phi^+$. Since $T^1$ normalizes each $I \cap U_{\alpha}$ this gives the surjectivity.

The compatibility of $\omega$ follows from the case $w=1$ by noting that any $p$-valuation is invariant under conjugation, see \cite[p.~169]{Sch11}.
\end{proof}

From now on we will assume ${\bf{G}}^{\text{der}}$ is simply connected and $p-1>eh$. We do not know whether Lemma \ref{weyl} continues to hold without the simply connectedness assumption. The discussion in \cite{Tit79} indicates that even the definition of parahoric subgroups is more subtle in that case.

\section{Rigid induction and slopes}\label{slp}

Start with a continuous character $\chi: T^1 \rightarrow K^{\times}$ and consider the conjugate $(w\chi)(t)=\chi(w^{-1}tw)$ for $w \in W$. Extend $w\chi$ to a character of 
$U_w^-T^1$ by making it trivial on $U_w^-$ and introduce the rigid analytic principal series representation
\begin{equation}\label{ind}
\Ind_{U_w^-T^1}^I(w\chi)=\{f \in \CC^{\rig}(I,K): f(bg)=(w\chi)(b)f(g) \y \forall b \in U_w^-T^1\}.
\end{equation}
This is a closed subspace of $\CC^{\rig}(I,K)$ because $|f(g)|\leq \|f\|_{\text{sup}}\leq \|f\|$ for all $g \in I$. Here $\|f\|_{\text{sup}}$ is the supremum norm of $f$ on $I$. The group $I$ acts by right translations. Of course, for the space in (\ref{ind}) to be non-trivial we require that $\chi$ is rigid when viewed as a $K$-valued function $T^1\rightarrow K^{\times}\hookrightarrow K$. In the sequel we will make this assumption: $\chi \in \CC^{\rig}(T^1,K)$.

\begin{lem}\label{restrict}
Let $\chi: T^1 \rightarrow K^{\times}$  be a rigid analytic character. Then the restriction map 
$f \mapsto f|_{U_w^+}$ gives an isomorphism of $U_w^+$-representations
$$
\Ind_{U_w^-T^1}^I(w\chi) \overset{\sim}{ \longrightarrow} \CC^{\rig}(U_w^+,K)
$$
preserving the Gauss norms. 
\end{lem}

\begin{proof}
By promoting the factorization in Lemma \ref{weyl} to an isomorphism of rigid spaces as in Remark \ref{promote} one can immediately extend a function in $\CC^{\rig}(U_w^+,K)$
to a function in $\CC^{\rig}(I,K)$ with the required transformation property. 

For the comment on norms, as Gauss norms are multiplicative it suffices to note that $\|\chi\|=1$. Indeed, we certainly have $\|\chi\| \geq \|\chi\|_{\text{sup}}=1$. Applying this to 
$\chi^{-1}$ we find that $\|\chi\|=1$.
\end{proof}

Next we transfer the $T^1$-action on $\Ind_{U_w^-T^1}^I(w\chi)$ across the isomorphism in Lemma \ref{restrict} to an action on $\CC^{\rig}(U_w^+,K)$. It is given by 
the formula
$$
(tf)(u)=(w\chi)(t)f(t^{-1}ut)
$$
for $t \in T^1$ and $u \in U_w^+$. In order to do explicit calculations we choose coordinates on $U_w^+$, but first we introduce a bookkeeping symbol.

\begin{defn}
For all roots $\beta$ we let $\varpi_{\beta}=\begin{cases}
\text{$\varpi$ if $\beta \in \Phi^-$} \\
\text{$1$ if $\beta \in \Phi^+$.} 
\end{cases}$
\end{defn}

\begin{notn}
With this symbol $(u_{w\alpha}(\varpi_{w\alpha}b_1),\ldots,u_{w\alpha}(\varpi_{w\alpha}b_{\ell}))$ is an ordered basis for $I \cap U_{w\alpha}$ for all $\alpha \in \Phi^+$. Recall that 
$(b_1,\ldots,b_{\ell})$ is a choice of $\Z_p$-basis for $\OO$, cf. the proof of Proposition \ref{char}. Letting $\alpha$ vary gives a basis for $U_w^+$. Once and for all we label the positive roots as $\Phi^+=\{\alpha_1,\ldots,\alpha_N\}$. Then our basis for $U_w^+$ is the string
$$
(u_{w\alpha_1}(\varpi_{w\alpha_1}b_1),\ldots,u_{w\alpha_1}(\varpi_{w\alpha_1}b_{\ell})) \y \ldots \y (u_{w\alpha_{N}}(\varpi_{w\alpha_{N}}b_1),\ldots,u_{w\alpha_{N}}(\varpi_{w\alpha_{N}}b_{\ell})).
$$
This lets us identify $\CC^{\rig}(U_w^+,K)$ with the space of rigid functions on the ball $\Z_p^{N\ell}$ in terms of the corresponding $\ell$-tuples of variables
$$
z_1=(z_{11},\ldots,z_{1\ell}) \y \ldots \y z_N=(z_{N1},\ldots,z_{N\ell}).
$$
Our power series expansions $f(z_1,\ldots,z_N)=\sum_I c_I z^I$ are indexed by tuples $I=(i_1,\ldots,i_N)$ where
$$
i_1=(i_{11},\ldots,i_{1\ell}) \y \ldots \y i_N=(i_{N1},\ldots,i_{N\ell}),
$$
and we adopt the following notation:
\begin{itemize}
\item $z^I=z_1^{i_1}\cdots z_N^{i_N}=z_{11}^{i_{11}}\cdots z_{N\ell}^{i_{N\ell}}$,
\item $|I|=|i_1|+\cdots+|i_N|=i_{11}+\cdots+i_{N\ell}$.
\end{itemize}
\end{notn}

By choosing a Chevalley basis we get a model for ${\bf{G}}$ over $\Z$. In particular ${\bf{G}}={\bf{G}}_0\times_{\Q_p}L$ for some connected reductive group 
${\bf{G}}_0$ split over $\Q_p$. Similarly ${\bf{T}}={\bf{T}}_0\times_{\Q_p}L$ for a split $\Q_p$-torus ${\bf{T}}_0 \subset {\bf{G}}_0$. We let $T_0={\bf{T}}_0(\Q_p)$. The maximal compact subgroup of $T_0$ will be denoted by $T_0^0$, and $T_0^1$ is the Sylow pro-$p$-subgroup of $T_0^0$. The main reason for passing to the $\Q_p$-groups ${\bf{G}}_0$ and
${\bf{T}}_0$ is that all roots $\alpha \in \Phi$ are $\Z_p^{\times}$-valued on $T_0^0$, and $\alpha: T_0^1 \rightarrow 1+p\Z_p$. In particular $tu_{\alpha}(x)t^{-1}=u_{\alpha}(x)^{\alpha(t)}$ for $t \in T_0^0$. Note that the right-hand side only makes sense when $\alpha(t) \in \Z_p$, which is why we have to put restrictions on $t$. 

We recommend the reader initially considers the case $L=\Q_p$ and $w=1$ at a first reading, which simplifies the notation tremendously. 

\bigskip

In these coordinates the $T_0^1$-action on $\CC^{\rig}(U_w^+,K)$ is given by
$$
(tf)(z_1,\ldots,z_N)=(w\chi)(t)\cdot f\big((w\alpha_1)(t^{-1})z_1,\ldots, (w\alpha_N)(t^{-1})z_N \big).
$$
We will differentiate this and get a $\Lie(T_0^1)$-module structure. Pick an $H \in \Lie(T_0^1)$ (there is no risk of confusion with the use of the letter $H$ in Section \ref{analytic}). 
The exponential map $\gamma: s \mapsto e^{sH}$ is defined on a small enough open ball $\Omega \subset \Z_p$ around $0$. This gives a locally $\Q_p$-analytic homomorphism 
$\gamma: \Omega \rightarrow T_0^1$ such that $\gamma(0)=1$ is the identity and $\gamma'(0)=\Lie(\gamma)(1)=H$. Thus $H$ sends $f$ to\footnote{We apologize for the double meaning of $s$ in this calculation.} 
\begin{align*}  
(Hf)(z_1,\ldots,z_N) &= \frac{d}{ds}(e^{sH}f)(z_1,\ldots,z_N)|_{s=0} \\  
 &= \frac{d}{ds} \big[(w\chi)(e^{sH})\cdot  f\big((w\alpha_1)(e^{-sH})z_1,\ldots, (w\alpha_N)(e^{-sH})z_N \big)\big]|_{s=0} \\
 &=d(w\chi)(H)f(z_1,\ldots,z_N)-\sum_{r=1}^N \sum_{s=1}^{\ell} d(w\alpha_r)(H)z_{rs}\frac{\partial f}{\partial z_{rs}}(z_1,\ldots,z_N),
\end{align*} 
by the product rule and the chain rule; see \cite[Lem.~9.5, Lem.~9.12]{Sch11}. Let us also emphasize that $d(w\chi)$ means the derivative at the identity of $w\chi$ viewed as a function $T_0^1 \rightarrow K$; see the definition in \cite[p.~59]{Sch11}. Similarly $d(w\alpha_r)$ is the derivative (at the identity) of the conjugate root $w\alpha_r$ viewed as a function $T_0^1 \rightarrow \Q_p^{\times}$.

\begin{lem}\label{eigenvalue}
The monomial $f_I(z_1,\ldots,z_N)=z^I$ is an eigenfunction for the $\Lie(T_0^1)$-action with eigensystem
$$
H \longmapsto d(w\chi)(H)-\sum_{r=1}^N d(w\alpha_r)(H)|i_r|.
$$
Any vector in $\Lie(T_0^1)$ acts continuously on $\CC^{\rig}(U_w^+,K)$.
\end{lem}

\begin{proof}
This is immediate from the above calculation. Indeed $z_{rs}\frac{\partial}{\partial z_{rs}}f_I=i_{rs}f_I$. Continuity follows since the eigenvalues of any $H$ are bounded uniformly in $I$ as 
$|i_r|$ is an integer.
\end{proof}

We think of $w \in W$ as being fixed in the entire discussion, and choose a cocharacter $\mu \in X_*({\bf{T}})$ adapted to $w$ in the sense that $\langle w\alpha,\mu\rangle>0$ for all 
$\alpha \in \Phi^+$. We view $\mu$ as a homomorphism $\mu: 1+p\Z_p \rightarrow T_0^1$ and introduce its derivative $H_{\mu}=\Lie(\mu)(1)$. A simple calculation shows 
$d(w\alpha)(H_{\mu})=\langle w\alpha,\mu\rangle$. In particular the operator 
$$
d(w\chi)(H_{\mu})\Id-H_{\mu}
$$
has eigenvalue $\lambda_I=\sum_{r=1}^N \langle w\alpha_r,\mu \rangle |i_r|$ on $f_I$ by Lemma \ref{eigenvalue}. (Obviously $\lambda_I$ depends on $w$ and $\mu$, but we think of those as being fixed, and we omit them to lighten the notation.) Note that $\lambda_I\in \N_0$, so it makes sense to consider its $p$-adic valuation $v(\lambda_I)$. 

We will explore the spectral theory of the operator $d(w\chi)(H_{\mu})\Id-H_{\mu}$ in more detail. First $C^{\rig}(U_w^+,K)$ has a slope decomposition
\begin{equation}\label{sd}
C^{\rig}(U_w^+,K)=C^{\rig}(U_w^+,K)^{< s} \oplus C^{\rig}(U_w^+,K)^{\geq s}
\end{equation}
for all $s \in \N_0$. Here $C^{\rig}(U_w^+,K)^{\geq s}$ is the subspace of functions $f(z_1,\ldots,z_N)=\sum_{I:v(\lambda_I)\geq s} c_Iz^I$. The space $C^{\rig}(U_w^+,K)^{< s}$ is defined similarly. We also consider the space $C^{\rig}(U_w^+,K)^{=s}$ spanned (topologically) by those monomials $z^I$ with $v(\lambda_I)=s$. 

In accordance with (\ref{sd}) we decompose a $K$-valued rigid function on $U_w^+$ as $f=f^{<s}+f^{\geq s}$ where $f^{<s}$ is the projection of $f$ onto 
$C^{\rig}(U_w^+,K)^{< s}$, and similarly for $f^{\geq s}$.

The next observation is inspired by the Hida idempotent in the theory of $p$-adic modular forms. This point of view is one of the main novelties of this paper.

\begin{lem}
For a given $s \in \N_0$ consider the decomposition 
$$
C^{\rig}(U_w^+,K)^{\geq s}=C^{\rig}(U_w^+,K)^{=s} \oplus C^{\rig}(U_w^+,K)^{\geq s+1}.
$$ 
Introduce the operator $\UU_s=\bigg(p^{-s}\big(d(w\chi)(H_{\mu})\Id-H_{\mu}\big)\bigg)^{p-1}$ on $C^{\rig}(U_w^+,K)$. Then
$$
\lim_{n\rightarrow \infty} \UU_s^{n!}
$$
defines the projection map $C^{\rig}(U_w^+,K)^{\geq s}\longrightarrow C^{\rig}(U_w^+,K)^{=s}$.
\end{lem}

\begin{proof}
Take some $f(z_1,\ldots,z_N)=\sum_{I:v(\lambda_I)\geq s} c_Iz^I$ belonging to $C^{\rig}(U_w^+,K)^{\geq s}$ and apply the $\UU_s$-operator,
$$
(\UU_sf)(z_1,\ldots,z_N)={\sum}_{I:v(\lambda_I)\geq s} c_I(p^{-s}\lambda_I)^{p-1}z^I.
$$
Doing this $n!$ times multiplies the coefficients $c_I$ by $(p^{-s}\lambda_I)^{(p-1)n!}$. If $v(\lambda_I)>s$ this clearly goes to $0$ as $n \rightarrow \infty$. 
On the other hand, if $v(\lambda_I)=s$ we have $(p^{-s}\lambda_I)^{p-1} \equiv 1$ modulo $p$ by Fermat. Therefore $(p^{-s}\lambda_I)^{(p-1)l}\rightarrow 1$ as $v(l)\rightarrow \infty$.
For $l=n!$ this amounts to letting $n \rightarrow \infty$. 
Consequently
$$
( \UU_s^{n!}f)(z_1,\ldots,z_N) \longrightarrow {\sum}_{I:v(\lambda_I)=s} c_Iz^I
$$
for $n \rightarrow \infty$ as desired.
\end{proof}

This yields a slope decomposition for closed $T_0^1$-invariant subspaces.

\begin{lem}
Let $V \subset C^{\rig}(U_w^+,K)$ be a closed $T_0^1$-invariant subspace. Then $V=V^{<s}\oplus V^{\geq s}$ for all $s \in \N_0$ where 
$$
V^{<s}=V \cap C^{\rig}(U_w^+,K)^{< s} \y \y \y \y \y V^{\geq s}=V \cap C^{\rig}(U_w^+,K)^{\geq s}.
$$
\end{lem}

\begin{proof}
It suffices to show $f \in V \Rightarrow f^{\geq s}\in V$ for all $s \in \N_0$. We use induction on $s$. The base case $s=0$ is trivial since $V=V^{\geq 0}$. Suppose $s \geq 0$ and the 
implication holds for $s$. Decompose $f^{\geq s}=f^{=s}+f^{\geq s+1}$. Here 
$$
f^{=s}=\lim_{n\rightarrow \infty} \UU_s^{n!}f^{\geq s}
$$
by the previous Lemma. By the induction hypothesis $f^{\geq s}\in V$. Thus each $\UU_s^{n!}f^{\geq s}\in V$ by $T_0^1$-invariance. Since $V$ is closed $f^{=s}\in V$. We conclude that
also $f^{\geq s+1} \in V$.
\end{proof}


We now consider closed $T_0^1U_w^+$-invariant subspaces.

\begin{prop}
Let $V \subset C^{\rig}(U_w^+,K)$ be a nonzero closed $T_0^1U_w^+$-invariant subspace. Then $V$ contains all the constant functions. 
\end{prop}

\begin{proof}
Pick a nonzero $h \in V$. Then $h(u_0)\neq 0$ for some $u_0 \in U_w^+$. In other words the right $u_0$-translate $f=u_0h \in V$ does not vanish at the identity. Therefore, when expressed in coordinates we have 
$$
f(z_1,\ldots,z_N)=\sum_I c_Iz^I
$$
with $c_0\neq 0$. (The subscript $0$ denotes the zero $N\ell$-tuple.) By the previous Lemma we known $f^{\geq s}\in V$ for all $s \in \N_0$. We claim $f^{\geq s}\rightarrow c_0$ 
as $s \rightarrow \infty$. Given $\epsilon>0$ we must find an $S$ such that $\forall I$:
$$
s \leq v(\lambda_I)< \infty \Longrightarrow |c_I|<\epsilon
$$
provided $s>S$. First pick an $S'$ such that $|I|>S' \Rightarrow |c_I|<\epsilon$, which is possible by rigidity. We will check that any $S$ satisfying
$$
p^S >S' \cdot {\max}_{\alpha \in \Phi^+} \langle w\alpha,\mu\rangle
$$
works. Indeed, suppose $s>S$ and $I$ is a nonzero tuple for which $p^s|\lambda_I$. Then
$$
p^S<p^s \leq \lambda_I=\sum_{r=1}^N \langle w\alpha_r,\mu\rangle |i_r| \leq |I|\cdot {\max}_{\alpha \in \Phi^+} \langle w\alpha,\mu\rangle
$$
and thus $|I|>S'$ by choice of $S$. This finishes the proof. 
\end{proof}

To summarize, every nonzero closed $T_0^1U_w^+$-invariant subspace of the rigid induction $\Ind_{U_w^-T^1}^I(w\chi)$ contains all the functions which are constant on $U_w^+$. In other words every such subspace must contain the function $\tilde{f}_0$ which extends the constant function $f_0=1$ on $U_w^+$. Note that $H\tilde{f}_0=d(w\chi)(H)\tilde{f}_0$ for all $H \in \Lie(T_0^1)$. In fact this holds for all $H \in \Lie(T^1)$ since any $t \in T^1$ acts by $(t\tilde{f}_0)(u)=(w\chi)(t)$. 
This observation will be used shortly in the next section where we address the $U_w^-$-action.

\section{Verma modules and irreducibility}\label{verma}

We now take into account the $U_w^-$-action which does not have a nice explicit description on $C^{\rig}(U_w^+,K)$. First we introduce some more notation. All our groups will be viewed as locally $\Q_p$-analytic groups:
\begin{itemize}
\item $\g=\Lie(I)\otimes_{\Q_p}K$
\item $\frak{t}=\Lie(T^1)\otimes_{\Q_p}K$
\item $\frak{t}_0=\Lie(T_0^1)\otimes_{\Q_p}K$
\item $\frak{u}_w^{\pm}=\Lie(U_w^{\pm})\otimes_{\Q_p}K$
\item $\frak{b}_w^{\pm}=\frak{t}\oplus \frak{u}_w^{\pm}$
\end{itemize}
We assume $K \supset L$. In particular this implies $\g_{\text{der}}$ is split in the sense of \cite[Sect.~1.10.1]{Dix96}. 

We view $d(w\chi)$ as a linear form 
$\frak{t}\rightarrow K$ and extend it to $\frak{b}_w^+$ by declaring it to be $0$ on $\frak{u}_w^+$. We denote the resulting $U(\frak{b}_w^+)$-module by $K_{d(w\chi)}$ and consider the Verma module 
$$
\mathcal{V}_{d(w\chi)}=U(\g)\otimes_{U(\frak{b}_w^+)}K_{d(w\chi)}.
$$
The rigid functions $C^{\rig}(I,K)$ form a $\g$-module. The map $X \mapsto X\tilde{f}_0$ factors through the Verma module and gives a nonzero $U(\g)$-linear map
$$
\psi: \mathcal{V}_{d(w\chi)} \longrightarrow \Ind_{U_w^-T^1}^I(w\chi).
$$
Indeed $H\tilde{f}_0=d(w\chi)(H)\tilde{f}_0$ for all $H \in \frak{t}$, and $X\tilde{f}_0=0$ for $X \in \frak{u}_w^+$ (as we are differentiating a constant function on $U_w^+$). 
This is a major contrast between this argument and the one in \cite{OS10} say. In the locally analytic setup one maps the Verma module to the {\it{dual}} of the induction, whereas $\psi$ maps $\mathcal{V}_{d(w\chi)}$ to the actual induction.

Since $\mathcal{V}_{d(w\chi)}=\bigoplus_{\lambda \in \frak{t}_0^*}\mathcal{V}_{d(w\chi)}[\lambda]$ is a sum of weight spaces the map $\psi$ clearly takes values in the subspace 
$\Ind_{U_w^-T^1}^I(w\chi)_{\frak{t}_0-\text{fin}}$ of $\frak{t}_0$-finite vectors (those $f$ for which $\text{span}_K\{Hf:H \in \frak{t}_0\}$ is finite-dimensional). 

\begin{lem}
The restriction map $f \mapsto f|_{U_w^+}$ gives an isomorphism of vector spaces
$$
\Ind_{U_w^-T^1}^I(w\chi)_{\frak{t}_0-\text{fin}} \overset{\sim}{\longrightarrow}C^{\alg}(U_w^+,K).
$$
(The right-hand side denotes the dense subspace of polynomial functions: $f(z_1,\ldots,z_N)=\sum_{I} c_Iz^I$ with $c_I=0$ for $|I|>\!\!>0$.)
\end{lem}

\begin{proof}
Let $f \in \Ind_{U_w^-T^1}^I(w\chi)_{\frak{t}_0-\text{fin}}$ and consider the finite-dimensional space $W=\text{span}_K\{Hf:H \in \frak{t}_0\}$. Decompose it into 
weight spaces $W=\bigoplus_{\lambda \in \frak{t}_0^*}W[\lambda]$, after possibly enlarging $K$. In order to show $f|_{U_w^+}$ is a polynomial function we may assume 
$f \in W[\lambda]$ for some $\lambda \in \frak{t}_0^*$. Expanding $f$ as $\sum_{I} c_If_I$ and applying $H \in \Lie(T_0^1)$ to it gives 
$$
\lambda(H)f=Hf=\sum_I c_I \big(d(w\chi)(H)-\sum_{r=1}^N d(w\alpha_r)(H) |i_r|\big) f_I.
$$
Comparing coefficients shows $c_I \neq 0 \Rightarrow \lambda(H)=d(w\chi)(H)-\sum_{r=1}^N d(w\alpha_r)(H) |i_r|$. For $H=H_{\mu}$ this last equation reads 
$$
\lambda(H_{\mu})-d(w\chi)(H_{\mu})=\lambda_I=\sum_{r=1}^N \langle w\alpha_r,\mu\rangle |i_r|.
$$
There are only finitely many tuples $I$ with this property. 

The map is obviously surjective: A polynomial function belongs to $C^{\rig}(U_w^+,K)_{\frak{t}_0-\text{fin}}$ since the $f_I$ are eigenfunctions for $\Lie(T_0^1)$.
\end{proof}

As a byproduct $\Ind_{U_w^-T^1}^I(w\chi)_{\frak{t}_0-\text{fin}}=\bigoplus_{\lambda \in \frak{t}_0^*}\Ind_{U_w^-T^1}^I(w\chi)_{\frak{t}_0-\text{fin}}[\lambda]$ decomposes as a direct sum of weight spaces, and 
$$
\dim_K \Ind_{U_w^-T^1}^I(w\chi)_{\frak{t}_0-\text{fin}}[\lambda]=\#\{I: d(w\chi) - \lambda=\sum_{r=1}^N d(w\alpha_r) |i_r|\}=\dim_K \mathcal{V}_{d(w\chi)}[\lambda].
$$
The last equality is part (ii) of \cite[Prop.~7.1.6, p. 233]{Dix96} applied to $w\Phi^+$ in place of $\Phi^+$, or rather a variant of the proof: As a vector space 
$\frak{g}=\frak{u}_w^-\oplus \frak{b}_w^+$ where $\frak{u}_w^-=\bigoplus_{r=1}^N \Lie(I \cap U_{-w\alpha_r})\otimes_{\Q_p}K$. In this decomposition each space $\Lie(I \cap U_{-w\alpha_r})$ is $\ell$-dimensional over $\Q_p$ and $\frak{t}_0$ acts on it via $-d(w\alpha_r)$. Inspecting the $\frak{t}_0$-action on a PBW basis for 
$\mathcal{V}_{d(w\chi)}\simeq U(\frak{u}_w^-)$ gives the result as in \cite{Dix96}.

\begin{prop}
If $\mathcal{V}_{d(w\chi)}$ is a simple $U(\g)$-module then $\psi$ is an isomorphism of $U(\g)$-modules
$$
\mathcal{V}_{d(w\chi)} \overset{\sim}{\longrightarrow} \Ind_{U_w^-T^1}^I(w\chi)_{\frak{t}_0-\text{fin}}
$$
and $\Ind_{U_w^-T^1}^I(w\chi)$ is topologically irreducible. Conversely, if $\Ind_{U_w^-T^1}^I(w\chi)$ is topologically irreducible the Verma module $\mathcal{V}_{d(w\chi)}$ is necessarily simple.  
\end{prop} 

\begin{proof}
The map $\psi$ is nonzero and $\g$-linear, therefore injective by the simplicity of $\mathcal{V}_{d(w\chi)}$. Since the dimensions match up the inclusions
$\mathcal{V}_{d(w\chi)}[\lambda] \hookrightarrow  \Ind_{U_w^-T^1}^I(w\chi)_{\frak{t}_0-\text{fin}}[\lambda]$ are bijective for all $\lambda \in \frak{t}_0^*$. Now take the direct sum. 

The irreducibility follows: Suppose $V \subset \Ind_{U_w^-T^1}^I(w\chi)$ is a nonzero closed $I$-invariant subspace. By the $T^1U_w^+$-invariance we know $V$ must contain 
the vector $\tilde{f}_0$, and therefore also all elements $X\tilde{f}_0$ for $X \in U(\frak{u}_w^-)$ by the $U_w^-$-invariance. From what we have just shown these $X\tilde{f}_0$
constitute all of $\Ind_{U_w^-T^1}^I(w\chi)_{\frak{t}_0-\text{fin}}$, which is a dense subspace (the polynomial functions on $U_w^+$ are dense in the space of rigid functions). Since $V$ is closed it must be all of $\Ind_{U_w^-T^1}^I(w\chi)$.

To show the converse we start by making the following observation.  Let $E \subset C^{\alg}(U_w^+,K)$ be an arbitrary $U(\frak{t})$-submodule with closure 
$\bar{E}\subset C^{\rig}(U_w^+,K)$. Then we assert that 
\begin{equation}\label{ray1}
E=\bar{E} \cap C^{\alg}(U_w^+,K).
\end{equation}
To see this suppose $f \in C^{\alg}(U_w^+,K)$ is a limit $f=\lim_{i\rightarrow \infty} f_i$ of functions $f_i \in E$. Let $\lambda \in \frak{t}_0^*$ and consider the projection map 
onto the $\lambda$-weight space $\text{pr}_{\lambda}: C^{\alg}(U_w^+,K) \rightarrow C^{\alg}(U_w^+,K)[\lambda]$. The target $C^{\alg}(U_w^+,K)[\lambda]$ is finite-dimensional, as we have seen above. Since $C^{\alg}(U_w^+,K)$ is topologized by the Gauss norm, and the $f_I$ are eigenfunctions, it is easy to see $\text{pr}_{\lambda}$ is continuous.
Therefore $\text{pr}_{\lambda}(f)=\lim_{i\rightarrow \infty} \text{pr}_{\lambda}(f_i)$. Observe that $E$ is a semisimple $U(\frak{t}_0)$-module since $C^{\alg}(U_w^+,K)$ is, and thus $E$ has a weight space decomposition $E=\bigoplus_{\lambda \in \frak{t}_0^*}E[\lambda]$. Hence all the terms $\text{pr}_{\lambda}(f_i)$ lie in $E[\lambda]$. Since $\dim_K E[\lambda]<\infty$ so does the limit $\text{pr}_{\lambda}(f)$ ($E[\lambda]$ is complete and therefore closed). Consequently $f=\sum_{\lambda} \text{pr}_{\lambda}(f) \in E$. This shows our assertion. 
In particular
\begin{equation}\label{ray2}
E \subsetneq C^{\alg}(U_w^+,K) \Longrightarrow \bar{E} \subsetneq C^{\rig}(U_w^+,K).
\end{equation}
Now suppose $\Ind_{U_w^-T^1}^I(w\chi)$ is topologically irreducible. Apply the above observation to $\im(\psi)$. Since
$\overline{\im(\psi)}=\Ind_{U_w^-T^1}^I(w\chi)$ we must have $\im(\psi)=\Ind_{U_w^-T^1}^I(w\chi)_{\frak{t}_0-\text{fin}}$. In other words $\psi: \mathcal{V}_{d(w\chi)} \rightarrow \Ind_{U_w^-T^1}^I(w\chi)_{\frak{t}_0-\text{fin}}$ is surjective. 
Again, since $\mathcal{V}_{d(w\chi)}[\lambda] \twoheadrightarrow \Ind_{U_w^-T^1}^I(w\chi)_{\frak{t}_0-\text{fin}}[\lambda]$ must be bijective $\forall \lambda$ as the two spaces have the same finite dimension, 
$\psi$ is bijective as well. Finally, to verify $\mathcal{V}_{d(w\chi)}$ is simple, suppose $\mathcal{E} \subsetneq \mathcal{V}_{d(w\chi)}$ is a $U(\g)$-submodule. Then, by the above implication applied to $E=\psi(\mathcal{E})$, we find that $\overline{\psi(\mathcal{E})} \subsetneq \Ind_{U_w^-T^1}^I(w\chi)$ and so $\overline{\psi(\mathcal{E})}=\{0\}$. Equivalently $\mathcal{E}=\{0\}$ by the bijectivity of $\psi$.  
\end{proof}

\begin{rem}
Assertion (\ref{ray1}) and the implication (\ref{ray2}) in the above proof are analogous to \cite[Lem.~3.13, Cor.~3.14]{Ray20b} respectively. Our proof is more or less just a rewording of the argument given there. 
\end{rem}

As explained in \cite[Cor.~3.4.11]{OS10} for example, $\mathcal{V}_{d(w\chi)}$ is simple for all $w \in W$ if (and only if) $\mathcal{V}_{d\chi}$ is simple. This is a straightforward argument; conjugation by $w$ identifies the Verma modules.

The BGG criterion \cite[Thm.~7.6.24]{Dix96} tells us $\mathcal{V}_{d\chi}$ is simple if and only if $(d\chi+\delta)(H_{\alpha})\notin \{1,2,3,\ldots\}$ 
for all positive roots $\alpha$ relative to the pair $(\frak{g},\frak{t})$. (To aid comparison, in the notation of \cite{Dix96} we have $\mathcal{V}_{d\chi}=M(d\chi+\delta)$.)

\section{Multiplicity one}

Suppose $\mathcal{V}_{d\chi}$ is simple. In this section we note that as $w \in W$ varies the representations $\Ind_{U_w^-T^1}^I(w\chi)$ are non-isomorphic. This relies on the non-existence of a Haar measure on the space of rigid analytic functions on $\Z_p$. We include the standard argument for convenience.

\begin{lem}\label{haar}
There is no nonzero translation-invariant continuous linear form
$C^{\rig}(\Z_p,K)\rightarrow K$.
\end{lem}

\begin{proof}
We let $T: C^{\rig}(\Z_p,K) \rightarrow C^{\rig}(\Z_p,K)$ denote the translation operator $(Tf)(x)=f(x+1)$. Suppose $\ell: C^{\rig}(\Z_p,K)\rightarrow K$ is a
translation-invariant form; meaning $\ell(Tf)=\ell(f)$. To show $\ell=0$ we consider the values of $\ell$ on the monomials $f_N(x)=x^N$. We will verify that 
$\ell(f_N)=0$ by strong induction on $N$. The case $N=0$ follows by observing $Tf_1=f_0+f_1$ and applying $\ell$. Now suppose $\ell(f_i)=0$ for all 
$i=0,\ldots,N-1$. Note that
$$
Tf_{N+1}=\sum_{i=0}^{N+1}\binom{N+1}{i}f_i=f_{N+1}+(N+1)f_N+\sum_{i=0}^{N-1}\binom{N+1}{i}f_i.
$$
Applying $\ell$ yields $\ell(Tf_{N+1})=\ell(f_{N+1})+(N+1)\ell(f_N)$ by the induction hypothesis. Since the left-hand side equals 
$\ell(f_{N+1})$ we conclude $\ell(f_N)=0$. By continuity $\ell=0$ on all of $C^{\rig}(\Z_p,K)$.
\end{proof}

Without any simplicity assumption on $\mathcal{V}_{d\chi}$ we have the following result.

\begin{prop}
$\Hom_I^{\text{cont}}\big(\Ind_{U_w^-T^1}^I(w\chi),\Ind_{U_{w'}^-T^1}^I(w'\chi)\big)=0$ for $w \neq w'$. 
\end{prop}

\begin{proof}
Rigid induction is easily checked to satisfy Frobenius reciprocity for globally analytic representations (one of the adjunction maps is given by the orbit map which is rigid analytic). 
So we have to show that 
$$
\Hom_{U_{w'}^-T^1}^{\text{cont}}\big(\Ind_{U_w^-T^1}^I(w\chi), K_{w'\chi}\big)=0.
$$
We identify $\Ind_{U_w^-T^1}^I(w\chi)$ with $C^{\rig}(U_w^+,K)$ as a $U_w^+$-representation. It suffices to show there is no nonzero $K$-linear form 
$\ell: C^{\rig}(U_w^+,K)\rightarrow K$ such that $\ell(uf)=\ell(f)$ for all $u \in U_w^+ \cap U_{w'}^-$. Note that this intersection is non-trivial since 
$w\Phi^+\cap w'\Phi^-\neq \varnothing$. Indeed $w^{-1}w'\neq 1$ has length $l(w^{-1}w')=|\Phi^+\cap w^{-1}w'\Phi^-|>0$. Pick a $\beta \in w\Phi^+\cap w'\Phi^-$
and consider the subgroup $I \cap U_{\beta}\subset U_w^+ \cap U_{w'}^-$. By ordering $\Phi^+$ appropriately we may factor $U_w^+$ as 
$U_w^+=Y\times (I \cap U_{\beta})$ and promote this to the level of rigid spaces as in remark \ref{promote}. Thus $C^{\rig}(U_w^+,K)\simeq C^{\rig}(Y,K)\hat{\otimes}_K C^{\rig}(I \cap U_{\beta},K)$ as representations of $I \cap U_{\beta}$. (The action on $C^{\rig}(Y,K)$ is trivial.)
This reduces the problem to checking there is no nonzero $I \cap U_{\beta}$-invariant linear form 
$C^{\rig}(I \cap U_{\beta},K)\rightarrow K$. Since $I \cap U_{\beta}\simeq \Z_p^{\ell}$ we just have to check there is no nonzero translation-invariant form
$C^{\rig}(\Z_p^{\ell},K)\rightarrow K$. This easily reduces to the case $\ell=1$ which is Lemma \ref{haar}.
\end{proof}

The previous argument is an adaptation of the proof of \cite[Prop.~3.5.1]{OS10} which they in turn partially attribute to Frommer. 

\section{Principal series}\label{ps}

To add some finishing strokes to the proof of our main theorem, suppose $\chi: T\rightarrow K^{\times}$ is a continuous character and inflate it to $B^-=\prod_{\alpha\in \Phi^-}U_{\alpha}\times T$. Consider the continuous induction
$$
\Ind_{B^-}^G(\chi)_{\text{cont}}=\{f \in C(G,K): f(bg)=\chi(b)f(g)\y \forall b \in B^-\}.
$$
By the Bruhat-Tits decomposition $G=\bigsqcup_{w\in W}B^-wJ$ it decomposes as a $J$-representation into
$$
\Ind_{B^-}^G(\chi)_{\text{cont}} \overset{\sim}{\longrightarrow} \bigoplus_{w\in W}\Ind_{J \cap wB^-w^{-1}}^J(w\chi)_{\text{cont}}
$$
by sending $f \mapsto (f_w)_{w\in W}$ where $f_w(g)=f(w^{-1}g)$. Since $J=T^0I$ and $T^1=T^0\cap I$ further restriction gives an isomorphism of 
$I$-representations 
$$
\Ind_{J \cap wB^-w^{-1}}^J(w\chi)_{\text{cont}} \overset{\sim}{\longrightarrow}
\Ind_{I \cap wB^-w^{-1}}^I(w\chi)_{\text{cont}}.
$$
The space of rigid vectors (or as in Emerton's terminology $\Bbb{I}$-analytic vectors, cf. \cite[Df.~3.3.1]{Eme17}) of the right-hand side coincides with the rigid induction
$\Ind_{U_w^-T^1}^I(w\chi)$ from above. This latter observation is a consequence of \cite[Prop.~3.3.7]{Eme17}. Our space $\CC^{\rig}(I,K)$ corresponds to $\CC^{\text{an}}(\Bbb{I},K)$
in his notation. 

Altogether this proves the main theorem \ref{main} in the introduction.

\section{The example $\Sp_4(\Q_p)$}\label{symp}

To highlight a concrete new example of our main result, we work out the case of $\Sp_4(\Q_p)$ explicitly. We follow the conventions in \cite{RS07}. Thus 
$$
G=\Sp_4(\Q_p)=\{g \in \GL_4(\Q_p): {^tg}Jg=J\}, \y \y \y
J=\begin{pmatrix} & & & 1 \\ & & 1 & \\ & -1 & & \\ -1 & & &\end{pmatrix}.
$$
We take our Borel subgroup $P$ to consist of all upper triangular matrices in $\Sp_4(\Q_p)$. The maximal torus is the subgroup of all diagonal matrices
$$
T=\{t_{a,b}=\begin{pmatrix} a& & &  \\ & b &  & \\ &  & b^{-1}& \\  & & & a^{-1}\end{pmatrix}: a,b \in \Q_p^{\times}\}.
$$
The Weyl group $W$ is dihedral of order $8$. There are two simple roots; $\Delta=\{\alpha,\beta\}$ where $\alpha(t_{a,b})=\frac{a}{b}$ and 
$\beta(t_{a,b})=b^2$. Then $\Phi^+=\{\alpha,\beta, \alpha+\beta, 2\alpha+\beta\}$. In particular $\theta=2\alpha+\beta$ is the highest root and the Coxeter number is $h=4$. We therefore require that $p>5$ in order to apply our theorem. Moreover, one immediately verifies that $\delta(t_{a,b})=a^2b$. See \cite[Ch.~2]{RS07} for a discussion of the root system. 

Let $\chi: T \rightarrow K^{\times}$ be a character. Write it as $\chi(t_{a,b})=\chi_1(a)\chi_2(b)$ for a pair of characters $\chi_i: \Q_p^{\times}\rightarrow K^{\times}$. We require that their restrictions $1+p\Z_p \rightarrow K^{\times}$ are rigid analytic. This means they can be expanded as
$\chi_i(\exp(px))=\sum_{r\geq 0}\gamma_{i,r}x^r$ for all $x \in \Z_p$. (Here $\gamma_{i,r}\in K$ converge to $0$ as $r \rightarrow \infty$.) If we write $\chi_i(t)=t^{c_i}$ for $t$ close to $1$, where $c_i=(d\chi_i)(1)\in K$, then $\gamma_{i,r}=\frac{1}{r!}p^rc_i^r \rightarrow 0$ precisely when
$$
v_p(c_i)>\frac{1}{p-1}-1=-\frac{p-2}{p-1}.
$$
Here $v_p$ denotes the valuation on $K$ with $v_p(p)=1$. We will assume this in the sequel. 

The BGG criterion for the simplicity of $\mathcal{V}_{d\chi}$ involves $H_{\alpha}\in \frak{t}$ for all positive roots $\alpha$. One can read those off from the coroots given in \cite[p.~42]{RS07}. We list them here for convenience.
$$
H_{\alpha}=\begin{pmatrix} 1 & & & \\ & -1 & & \\ & & 1 & \\ & & & -1\end{pmatrix}
H_{\beta}=\begin{pmatrix} 0 & & & \\ & 1 & & \\ & & -1 & \\ & & & 0\end{pmatrix}
H_{\alpha+\beta}=\begin{pmatrix} 1 & & & \\ & 1 & & \\ & & -1 & \\ & & & -1\end{pmatrix}
H_{2\alpha+\beta}=\begin{pmatrix} 1 & & & \\ & 0 & & \\ & & 0 & \\ & & & -1\end{pmatrix}.
$$
This allows us to make the BGG criterion explicit in terms of $c_1$ and $c_2$. 
\begin{itemize}
\item $(d\chi+\delta)(H_{\alpha})=c_1-c_2+1 \notin \Z_{>0}$;
\item $(d\chi+\delta)(H_{\beta})=c_2+1 \notin \Z_{>0}$;
\item $(d\chi+\delta)(H_{\alpha+\beta})=c_1+c_2+3 \notin \Z_{>0}$;
\item $(d\chi+\delta)(H_{2\alpha+\beta})=2c_1+2 \notin \Z_{>0}$.
\end{itemize}
All four conditions are satisfied precisely when the Verma module $\mathcal{V}_{d\chi}$ is simple. If so, our main result shows that 
$\Ind_P^G(\chi)_{\Bbb{I}-\text{an}}$ breaks up as a direct sum of eight mutually non-isomorphic topologically irreducible representations 
$\Ind_{I \cap wPw^{-1}}^I(w\chi)_{\Bbb{I}-\text{an}}$ parametrized by the Weyl elements $w$. Here the pro-$p$ Iwahori subgroup $I$ consist of symplectic matrices which are lower triangular modulo $p$,
$$
I=\{g \in \Sp_4(\Z_p): g \equiv \begin{pmatrix} 1 & & & \\ * & 1 & & \\ *& *& 1 & \\ *& *& *& 1\end{pmatrix} \text{mod $p$}\}.
$$
When $w=1$ the subgroup $I \cap wPw^{-1}$ consists of upper triangular matrices $\begin{pmatrix} a& *& *& * \\ & b &*  &* \\ &  & b^{-1}&* \\  & & & a^{-1}\end{pmatrix}$
in $\Sp_4(\Z_p)$ with $a,b \in 1+p\Z_p$ and all $\ast$-entries in $p\Z_p$. 


\subsection*{Acknowledgments} 

We are grateful to Matthias Strauch for sharing his expertise with us at various stages of this project. Thanks are also due to Jishnu Ray for some initial guidance and encouragement. 
Thanks to both for extensive feedback on a preliminary draft. 



\bigskip

\noindent {\it{E-mail addresses}}: {\texttt{arlahiri@ucsd.edu}, {\texttt{csorensen@ucsd.edu}}

\noindent {\sc{Department of Mathematics, UCSD, La Jolla, CA, USA.}}

\end{document}